\documentclass[11pt]{amsart}

\usepackage{amssymb,amsmath,amsthm,enumerate,latexsym, graphics,shapepar,nccmath,mathtools}
\usepackage{tikz-cd} 
\usepackage[all,2cell,dvips]{xy}
\usepackage[utf8]{inputenc}
\usepackage[pagebackref, colorlinks=true, pdfstartview=FitV, linkcolor=red, citecolor=blue, urlcolor=blue]{hyperref}
\usepackage{cleveref}
\usepackage[margin=0.9in]{geometry}

\usepackage{parskip}
\usepackage{enumitem}
\newtheorem{thm}{Theorem}[section]
\newtheorem{lemma}[thm]{Lemma}
\newtheorem{rmk}[thm]{Remark}
\newtheorem{prop}[thm]{Proposition}
\newtheorem{cor}[thm]{Corollary}
\newtheorem{defn}[thm]{Definition}
\newtheorem{question}[thm]{Question}
\newtheorem{example}[thm]{Example}

\newtheorem{notation}[thm]{Notation}
\newtheorem{conjecture}[thm]{Conjecture}

\newcommand{\conv}{\mbox{\rm conv}}

\newcommand{\LCM}{\mbox{\rm LCM}}
                     
\newcommand{\ann}{\mbox{\rm ann}}

\newcommand{\tor}{\mbox{\rm Tor}}

\newcommand{\depth}{\mbox{\rm depth}}

\newcommand{\pdim}{\mbox{\rm pdim}}
\newcommand{\reg}{\mbox{\rm reg}}

\newcommand{\hgt}{\mbox{\rm ht}}

\DeclareMathOperator{\codim}{codim}
           \newcommand{\supp}{\mbox{\rm Supp}}

\newcommand{\sk}{\mbox{$\mathsf{k}$}}

\begin{document}

\title[A comparison of regularity of monomial ideals and their integral closure]{A Comparison of the Regularity of Certain Classes of Monomial Ideals and Their Integral Closures}

\author{Omkar Javadekar}
\address{Chennai Mathematical Institute, Siruseri, Tamilnadu 603103. India}
\email{omkarj@cmi.ac.in}



\begin{abstract}
  Let $S = \mathsf{k}[x_1, \ldots, x_n]$,  $I$ be an ideal of $S$, and  $\bar{I}$ denote its integral closure. A conjecture of Küronya and Pintye states that for any homogeneous ideal $I$ of $S$, the inequality $\operatorname{reg}(\bar{I}) \leq \operatorname{reg}(I)$ holds, where $\operatorname{reg}(\_)$ denotes the Castelnuovo--Mumford regularity. In this article, we prove the conjecture for certain classes of monomial ideals.
\end{abstract}

\maketitle

\section{Introduction}

The computation of the Castelnuovo--Mumford regularity of various classes of homogeneous ideals over a polynomial ring has been an active area of research in recent years. Although the minimal graded free resolutions of homogeneous ideals in a polynomial ring are finite, computing these resolutions, or even the weaker task of determining the (graded) Betti numbers and the Castelnuovo--Mumford regularity, remains a complex problem. Since monomial ideals form a large and well-structured subclass of homogeneous ideals, it is common for researchers to first establish results for monomial ideals before addressing the more general case.

The main problem of interest in this article is the following conjecture, due to K\"{u}ronya and Pintye~\cite{KP13}:

\begin{conjecture}\label{conj:integral-closure}
    Let $I$ be a homogeneous ideal of $S = \sk[x_1, \ldots, x_n]$. If $\bar{I}$ denotes the integral closure of $I$, then
    \[
        \reg(\bar{I}) \leq \reg(I).
    \]
\end{conjecture}

We provide a positive answer to the above conjecture for certain classes of monomial ideals of $S$. To the best of our knowledge, not much is known about the conjecture. The classes of ideals for which it has been studied primarily include monomial ideals arising as powers of edge ideals of certain graphs, powers of some square-free monomial ideals, and edge ideals of weighted oriented graphs. More details can be found in \cite{MV22, MVZ24}. In a related article \cite{kumar2021regularity}, a comparison between the regularities of ordinary powers, symbolic powers, and the integral closures of powers of certain classes of ideals is carried out. In \cite{fakhari2024some}, several inequalities concerning the regularity of integral closures are obtained. In \cite{swanson2023differences}, it is shown that the regularity of the second power of square-free monomial ideals can be arbitrarily larger than the regularity of their integral closures.

In this work, we study Conjecture~\ref{conj:integral-closure} for the following families of ideals:
\begin{enumerate}[label={\rm (\alph*)}]
    \item Monomial ideals in $\sk[x, y]$ and $\sk[x, y, z]$
    \item Complete intersection monomial ideals
    \item Gorenstein monomial ideals of height $3$
    \item Stable monomial ideals
\end{enumerate}

We show that Conjecture~\ref{conj:integral-closure} has a positive answer for the ideals listed in (b), (c), and (d) above, as well as for all monomial ideals in $\sk[x, y]$. We also show that Conjecture~\ref{conj:integral-closure} holds for all monomial ideals in $\sk[x, y, z]$ if and only if it holds for height $2$ monomial ideals in $\sk[x, y, z]$. Along the way, we also establish some results about the Betti numbers of $\bar{I}$ for the aforementioned families of ideals.

As mentioned earlier, the computation of $\reg(I)$ is a rather difficult task. Moreover, the homological properties of $\bar{I}$ are not well understood. This makes Conjecture~\ref{conj:integral-closure} quite hard to tackle. The idea behind choosing the classes (a)--(d) is that one has at least some information about either $\reg(I)$ or the generating set of the integral closure $\bar{I}$. For instance, for stable monomial ideals $I$, one has an explicit resolution due to Eliahou--Kervaire~\cite{EK90}, and hence an explicit formula for their regularity, which is equal to the maximum degree of a minimal generator of $I$. 

We also note that commonly used techniques, such as passing to polarization, do not help here, as the integral closure operation does not behave well with polarization. In conclusion, at present, we do not have any general techniques to resolve Conjecture~\ref{conj:integral-closure}.

The organization of the article is as follows. \Cref{sec:prelim} contains the necessary definitions and known results. In \Cref{sec:2-3var}, we study Conjecture~\ref{conj:integral-closure} for monomial ideals in $\sk[x, y]$ and $\sk[x, y, z]$. Sections~\ref{sec:CI}, \ref{sec:Gor}, and \ref{sec:Stable} are devoted to proving Conjecture~\ref{conj:integral-closure} for complete intersection monomial ideals, height $3$ Gorenstein monomial ideals, and stable monomial ideals, respectively. Finally, in \Cref{sec:misc}, we make a few remarks about possible directions for further study and suggest additional classes of ideals for which the validity of the conjecture could be tested.

\section*{Acknowledgments}

The author would like to thank his advisor, H.~Ananthnarayan, for his support during the Ph.D. program and for several valuable discussions on graded free resolutions. A significant portion of this work was carried out during the author’s Ph.D. years. The author is currently supported by a Postdoctoral Fellowship at the Chennai Mathematical Institute, and he gratefully acknowledges additional support from the Infosys Foundation.

\section{Preliminaries}\label{sec:prelim}
Throughout this article, we use $S$ to denote the polynomial ring $\sk[x_1,\ldots, x_n]$ over a field $\sk$. Also, for convenience, while dealing with cases $n=2,3$, we shall use variables $x,y,z$ instead of $x_1,x_2,x_3$.

We start by recalling some basic definitions and facts about graded free resolutions. For more details on the same, we refer the reader to \cite{PeevaBook}.
\begin{defn}\label{defn:graded-stuff} Consider the standard graded polynomial ring $S=\sk[x_1,\ldots, x_n]$ with $\deg(x_i)=1$ for all $i$. Let $M$ be a finitely generated $\mathbb Z$-graded $S$-module. 
    \begin{enumerate}[label={\rm (\alph*)}]
        \item The $i^{th}$ Betti number of $M$, denoted $\beta_i(M)$, is defined as
        $$\beta_i(M)=\dim_{\mathsf k}\left(\tor_i^S(M,\sk)\right).$$
        \item The $(i,j)^{th}$ graded Betti number of $M$, denoted $\beta_{i,j}(M)$ is defined as $$\beta_{i,j}(M)=\dim_{\mathsf k}\left(\tor_i^S(M,\sk)_j\right).$$
        \item When $M\neq 0$, the (Castelnuovo--Mumford) regularity of $M$, denoted $\reg(M)$ is defined as 
        $$\reg(M)= \max\{j-i\mid \beta_{i,j}(M) \neq 0\}.$$
        Note that by the well-known Hilbert syzygy theorem, the maximum exists.
    \end{enumerate} 
\end{defn}

An ideal $I \subseteq S$ is called \emph{graded} (or \emph{homogeneous}) if it has a minimal generating set consisting of homogeneous polynomials. In this article, we use the terms \emph{graded ideal} and \emph{homogeneous ideal} interchangeably.
 A \emph{monomial ideal} of $S$ is an ideal that has a generating set consisting of monomials. Every monomial ideal $I$ has a unique minimal monomial generating set, and we denote it with $\mathcal G(I)$. In this article, we shall be working primarily with the invariants introduced in \Cref{defn:graded-stuff} above for $M=I$, a monomial ideal of $S$. 

One can also define a finer grading on $S$, often called multigrading or $\mathbb Z_{\geq 0}^n$-grading,  where the graded components correspond to the set of all monomials in $S$. In this case, we set degree of
the monomial $x_1^{a_1}x_2^{a_2}\cdots x_n^{a_n}$ to be $(a_1,a_2,\ldots,a_n)$. Then every $\mathbb Z^n$-graded $S$-module $M$ has a $\mathbb{Z}^n$-graded minimal free resolution. 
Given any $i \geq 0$ and $C\in \mathbb Z^n$, we thus define $$\beta_{i,C}(M)= \dim_{\mathsf k}\left(\left(\tor_i^S(M, \sk\right)_{C}\right).$$

Given any $A=(a_1,\ldots, a_n)\in \mathbb Z_{\geq 0}^n$, for the monomial $x^A=x_1^{a_1}x_2^{a_2}\cdots x_n^{a_n}$, we define $\supp(x^A)=\{i \mid a_i>0\}$.

For $A=(a_1,\ldots,a_n), B=(b_1,\ldots,b_n)$ with $a_i, b_j \in \mathbb R_{\geq 0}$, we say that $A \preccurlyeq B$ if $a_i \leq b_i$ for all $i$. Note that for monomials $x^A, x^B$, we have $x^A \mid x^B$ if and only if $A \preccurlyeq B$.

\begin{rmk}\label{rmk:lcmIsMaxDegree}
    Let $I$ be a monomial ideal with the minimal monomial generating set $\mathcal G(I)$, and $x^B\coloneqq\LCM\{x^A \mid x^A \in \mathcal G(I)\}$. Then from the Taylor's resolution, we see that $\beta_{i,C}(S/I)=0$ if $C \not \preccurlyeq B$.
\end{rmk}

The following serves as a key tool in proving \Cref{conj:integral-closure} for complete intersection and height 3 Gorenstein monomial ideals. 
\begin{rmk}\label{rmk:upToCodim}
    Let $S=\sk[x_1,\ldots, x_n]$, and $M$ a finitely generated graded $S$-module. Then the function $i \mapsto \max\{j \mid \beta_{i,j}^R(M) \neq 0\}$ is strictly increasing for $ 0 \leq i \leq \codim(M)\coloneqq \hgt(\ann_S(M))$.
\end{rmk}

We now recall the definition of integral closure followed by facts and known results about them. We refer the interested reader to \cite{HS06} for more details on integral closures. 
\begin{defn}
    Let $R$ be a Noetherian ring, and $I\subseteq R$ be an ideal. An element $x\in R$ is said to be \emph{integral over $I$} if there exist $n \in \mathbb N$, and $a_1,\ldots, a_n$, with $a_i\in I^i$ such that 
    \[ x^n + a_1 x^{n-1}+ \cdots + a_{n-1}x+a_n=0.\]
The set of all integral elements over $I$ forms an ideal. We call it the \emph{integral closure of $I$}, and denote it by $\bar{I}$.
 An ideal $I$ is said to be \emph{integrally closed} if $I=\bar I$.
\end{defn}

\begin{rmk}[\cite{HS06}, Section 1.4] Let $S=\sk[x_1,\ldots,x_d]$, and $I$ be a monomial ideal of $S$. Then
\begin{enumerate}[label={\rm (\alph*)}]
    \item $\bar I$ is a monomial ideal. 
    \item  A monomial $f\in \bar I$ if and only if there exists $r\in\mathbb N$, and monomials $f_1,\ldots,f_r\in I$ such that $f^r=\prod\limits_{i=1}^r f_i$.
\end{enumerate}
\end{rmk}
Thus, when $I$ is a monomial ideal, testing whether a given monomial belongs to the integral closure or not is relatively easy. 
 We will be using the part (b) of the remark above throughout the article. 

Given any subset $U$ of $\mathbb R^n$, let $\conv(U)$ denote the convex hull of $U$ in $\mathbb R^n$. The \emph{exponent vector} of a monomial $x_1^{a_1}x_2^{a_2}\cdots x_n^{a_n}$ is defined as $(a_1,\ldots,a_n)$. If $I$ is a monomial ideal of $S$, then the set $\{A \in \mathbb Z_{\geq 0}^n \mid x^A \in I\}$ is called the \emph{exponent set of $ I$}. 
The convex hull of all exponent vectors of $I$, denoted $NP(I)$, is called the \emph{Newton polyhedron of $I$}. We denote by $\mathcal V(I)$ the minimal subset of $\mathcal G(I)$ such that $NP(I)=\conv(\mathcal V(I))+ \mathbb R^n_{\geq 0}$, and call it the set of \emph{extreme points} or \emph{corner points} of $NP(I)$.  We define $\delta(I)=\max\{\deg(x^A)\mid A \in \mathcal V(I)\}$.

If $I$ is a monomial ideal, there is a nice geometric description of $\bar I$. We recall it in the following remark, since it will be used later sections.  

\begin{rmk}[\cite{HS06}, Proposition 1.4.6]\label{rmk:convexhullinterpretation}{\rm 
      For a monomial ideal $I$, the exponent set of $\bar I$ equals the set of all integer vectors that belong to the convex hull of exponent vectors of $I$.
}\end{rmk}

In \cite{Ho22} Hoa provided bounds for the regularity of powers of $\bar I$ in terms of the height of $I$. 
\begin{rmk}[cf.~\cite{Ho22}, Theorem 2.7]\label{thm:Hoa-bound-integral-closure}
    Given any monomial ideal $I\subseteq S$, and for all $m \in \mathbb N$, we have 
    $$\delta(I) m \leq \reg(\overline{ I^m}) \leq \delta(I) m+ \dim(S/I). $$
\end{rmk}
Observe that when $m = 1$, the above result implies that $\delta(I) \leq \reg(\bar{I}) \leq \delta(I) + \dim(S/I)$, which provides strong bounds, particularly when $\dim(S/I)$ is small.

\section{Monomial ideals in \(\sk[x,y]\) and \(\sk[x,y,z]\)}\label{sec:2-3var}

We begin with some general results that hold over arbitrary polynomial rings. In the next proposition, we note that that the Conjecture \ref{conj:integral-closure} holds if $I$ is $\langle x_1, \ldots, x_n\rangle$-primary.

\begin{prop}\label{prop:height-n-reg}
    Let $I\subseteq S$ be a graded ideal such that $\sqrt{I}= \langle x_1, \ldots, x_n\rangle$. Then $\reg(\bar I)\leq \reg(I)$.
\end{prop}
\begin{proof}
    Since $\sqrt{I}=\langle x_1, \ldots, x_n\rangle$, we see that $\dim(S/I)=0$. Thus, taking $m=1$ in Remark \ref{thm:Hoa-bound-integral-closure}, we get that $\delta(I)=\reg(\bar I)$. Since $I$ has a minimal generator in degree $\delta(I)$, we see that $\reg(I)\geq \delta(I)$, which gives the required inequality.
\end{proof}

In light of the proposition above, we may always assume that $\hgt(I)<n$. Note that since we are working over a polynomial ring, which is an integral domain, each nonzero ideal is of height at least one. Observe that for $S=\sk[x,y]$, this says that we only need to consider the height one ideals. With the help of the following proposition, we get that as far as \Cref{conj:integral-closure} is concerned, it is enough to consider ideals of height $\geq 2$. 

\begin{prop}\label{prop:reducing-hgt-1-to-hgt-at-least-2}
    Let $I$ be a monomial ideal of $S=\sk[x_1,\ldots, x_n]$. If $\hgt(I)=1$, then there exists a variable $x_j$ such that $I=x_j I_1$ and $\bar I=x_j \overline{I_1}$. In particular, $\reg(\bar I)\leq \reg(I)$ if and only if $\reg(\overline{I_1})\leq \reg(I_1)$.
\end{prop}
\begin{proof}
    We first note that if a variable $x_j$ as in the statement exists, then $\reg(I)=1+\reg(I_1)$ and $\reg(\bar I)= 1+ \reg(\overline{I_1})$; and the ``In particular'' part of the statement follows.

    Let us now prove the existence of $x_j$. Since $\hgt(I)=1$, there exists a minimal prime of height one over $I$. Since $I$ is a monomial ideal, such minimal prime must be $\langle x_j\rangle$ for some $1\leq j \leq n$. Thus, every element of $\mathcal{G}(I)$ is divisible by $x_j$.  Thus, if $I=x_jI_1$, then  $NP(I)=e_j + NP(I_1)$, and hence $\bar I = x_j \overline{I_1}$.
\end{proof}

Observe that if the ideal $I_1$ in the above proposition has height one, then the same idea applies to $I_1$, which yields $I_2$ with the property that $I_1=x_j I_2$ for some variable $x_j$. Since $I$ is finitely generated, this process stops after a finite number of steps (say after $t$-steps), and we get an ideal $I_t$, which is either the unit ideal $S$, or has height at least $2$. 
As an immediate corollary of this observation, we get the following:

\begin{cor}
Let $I$ be a monomial ideal of $S=\sk[x,y]$. Then $\reg(\bar I) \leq \reg(I)$. 
\end{cor}
\begin{proof}
    We may assume that $I\neq 0$, or equivalently $1\leq \hgt(I)\leq 2$. By Proposition \ref{prop:reducing-hgt-1-to-hgt-at-least-2}, as noted earlier, we may assume that $\hgt(I)=2$. Hence, $\sqrt{I}=\langle x, y\rangle$, and we are done by Proposition \ref{prop:height-n-reg}. 
\end{proof}

Combining \Cref{prop:height-n-reg} with the discussion before the corollary above, the following is immediate. 
\begin{cor} \Cref{conj:integral-closure} holds for all monomial ideals in $\sk[x,y,z]$ if and only if it holds for all monomial ideals $I\subseteq \sk[x,y,z]$ with $\hgt(I)=2$.
\end{cor}

In view of \Cref{thm:Hoa-bound-integral-closure}, we see that for a monomial ideal $I\subseteq \sk[x,y,z]$ of height 2, we obtain the inequality
$\delta(I) \leq \reg(\bar I)\leq \delta(I)+1$. Thus, we see that \Cref{conj:integral-closure} holds for all equigenerated monomial ideals in $\sk[x,y,z]$ if the answer to the following question is positive: 
\begin{question}
    Let $I\subseteq \sk[x,y,z]$ be an equigenerated monomial ideal of height 2 having a linear resolution. Does $\bar I$ have a linear resolution?
\end{question}

We end this section by showing that for a monomial ideal $I\subseteq \sk[x,y]$, the total Betti numbers of $I$ are less than or equal to the total Betti numbers of $\bar I$.

\begin{prop}
    Let $I$ be a monomial ideal in $\sk[x,y]$. Then $\mu(I)\leq \mu(\bar I)$. In particular, we have $\beta_i(S/I) \leq \beta_i(S/ \bar I)$ for $0 \leq i \leq 2$.
\end{prop}
\begin{proof}
    If $\mu(I)\leq 1$, then there is nothing to prove. So, we assume that $\mu(I)\geq 2$. Then we have $\pdim(S/I)=2$. Hence, in view of the Hilbert--Burch theorem, it is enough to prove that $\mu(I)\leq \mu(\bar I)$. 
    
    Let $\mathcal G(I)= \{x^{a_i}y^{b_i}\mid 1 \leq i \leq k\}$, with $a_1>a_2>\cdots>a_k$ and $b_1<b_2<\cdots<b_k$. Let $m_{ij}$ denote the absolute value of the slope of the line segment joining $(a_i,b_i)$ and $(a_j,b_j)$.  
    Define $i_1=1$, and let $i_2$ be the least integer such that $m_{i_1i_2}\leq m_{i_1 j}$ for all $1< j \leq k$. Let $i_3$ be the least integer such that $m_{i_2i_3}\leq m_{i_2j}$ for all $i_2< j \leq k$. Continuing in this way, we get a subset $S=\{ x^{a_{i_j}}y^{b_{i_j}} \mid 1\leq j\leq r\}$ of $\mathcal G(I)$. 
    Note that by Remark \ref{rmk:convexhullinterpretation}, we have $S\subseteq \mathcal G(\bar I)$. 

    To show that $\mu(I)\leq \mu(\bar I)$, we construct an injective function from $\mathcal G(I)$ to $\mathcal G(\bar I)$ which is identity on $S\subseteq G(I)\cap \mathcal G(\bar I)$.
     To do this, for each $j$, it is enough to define an injective function  
    $$\varphi: \{x^{a_{i_j}}y^{b_{i_j}}, x^{a_{{i_j}+1}}y^{b_{{i_j}+1}}, \ldots, x^{{a_{i_{j+1}}}}y^{{b_{i_{j+1}}}}\} \longrightarrow  \{x^ry^s \mid a_{i_j} \geq r \geq a_{i_{j+1}}\} \cap \mathcal G(\bar I),$$ which fixes $x^{a_{i_j}}y^{b_{i_j}}$ and $x^{a_{i_{j+1}}}y^{b_{i_{j+1}}}$.
    To do this, without loss of generality, we may assume that $b_{i_j}=0$ and $a_{i_{j+1}}=0$, and also that $b_{i_{j+1}}\geq a_{i_j}$. Define the function $\varphi$ as follows \[\varphi(x^ry^s)= x^ry^{{b_{i_{j+1}} - }\left\lfloor{\dfrac{b_{i_{j+1}}}{a_{i_j}}r}\right\rfloor}.\] 
    By Remark \ref{rmk:convexhullinterpretation}, the image of $\varphi$ is indeed a subset of $\{x^ry^s \mid a_{i_j} \geq r \geq a_{i_{j+1}}\} \cap \mathcal G(\bar I)$.
    Since the exponent of $x$ is fixed by $\varphi$, it is injective. Moreover, since $b_{i_{j+1}}\geq a_{i_j}$, the image of $\varphi$ is a minimal generating set of the ideal that it generates. This shows that $\varphi$ has the required properties, completing the proof.
\end{proof}

\section{Complete Intersection Ideals}\label{sec:CI}
This section addresses \Cref{conj:integral-closure} for complete intersection monomial ideals. These are the ideals whose generating set forms a regular sequence in $S$. Being monomial ideals, such ideals have quite a simple structure. We note it in the following remark.

\begin{rmk}\label{rmk:regsequence} \hfill{}
{\rm
     \begin{enumerate}[label={\rm (\alph*)}]
        \item A sequence of monomials $x^{A_1},\ldots,x^{A_c}$ is a regular sequence in $S$ if and only if the monomials are pairwise coprime. In other words, the set of variables appearing in $x^{A_i}$ and $x^{A_j}$ are disjoint. From this fact, together with Remark \ref{rmk:convexhullinterpretation}, it is easy to see that if $I$ is a monomial ideal generated by a regular sequence, then each minimal generator of $I$ is also a minimal generator of $\bar I$. Hence, $\mu(I) \leq \mu(\bar I)$.
       
        \item If $I=\langle x^{A_1},\ldots,x^{A_c}\rangle$, where $x^{A_1},\ldots,x^{A_c}$ is a regular sequence, then the Koszul complex on these generators gives a minimal free resolution of $S/I$. In particular, we have $\beta_i(S/I)=\binom{c}{i}$ for all $i \geq 0$.
     \end{enumerate}
}\end{rmk}

The following is a well-known result (e.g. see \cite[Theorem 4.10]{BG21}) that gives lower bounds on the Betti numbers of a mononial ideal in terms of its height. We present its proof, since it is short, and also for the sake of completeness.
\begin{thm}[\cite{BG21}, Theorem 4.10]\label{thm:lowerboundson BettinumbersMono}
Let $I$ be a monomial ideal of $\sk[x_1,\ldots,x_n]$ with $\hgt(I)=c$. Then for every $i \geq 0$, we have $\beta_i(S/I)\geq \binom{c}{i}$.
\end{thm}

\begin{proof}
    By passing to polarization, we may assume that $I$ is a squarefree ideal. Then $I$ has a primary decomposition of the form $I=\bigcap_{i} \mathfrak{p}_i$, where each $\mathfrak{p}_i$ is a prime ideal generated by variables and has height at least $c$. 

    Let $\mathfrak p_1=\langle x_1,\ldots, x_{c'}\rangle$. Then we have $ \beta_i(S/I) \geq \beta_i(S_{\mathfrak{p}_1} / I_{\mathfrak p_1}) = \binom{c'}{i}$ for all $i$. Since $c'\geq c$, we get $\beta_i(S/I)\geq \binom{c}{i}$ for all $i\geq 0$.
\end{proof}

Since $\hgt(I)=\hgt(\bar I)$, as a consequence of Remark \ref{rmk:regsequence} and Theorem \ref{thm:lowerboundson BettinumbersMono}, we get the following: 

\begin{cor}
    If $I\subseteq S=\sk[x_1,\ldots, x_n]$ is a complete intersection ideal, then $\beta_i(S/\bar I) \geq \beta_i(S/ I)=\binom{c}{i}$ for all $i \geq 0$.
\end{cor}

The next result shows that Conjecture \ref{conj:integral-closure} holds for complete intersection ideals.

\begin{thm}\label{thm:CI-reg}
    Let $I\subseteq S=\sk[x_1,\ldots,x_n]$ be a complete intersection ideal. Then $\reg(\bar I) \leq \reg(I)$.
\end{thm}
\begin{proof}
    Let $\mathcal G(I)= \{ x^{A_1},\ldots,x^{A_c}\}$. Since $I$ is a complete intersection, $\supp(x^{A_i})\neq \supp(x^{A_j})$ for $i \neq j$, and $x^{A_1},\ldots,x^{A_c}$ be a regular sequence. Set $d_i = \deg(x^{A_i})$. Since the Koszul complex on $x^{A_1},\ldots,x^{A_c}$ resolves $S/I$, we have $\reg(S/I)= -c + \sum_{i=1}^{c} d_i$. 
    Note that since $c=\hgt(I)=\hgt(\bar I)$, by Remark \ref{rmk:upToCodim}, we get that $\reg(S/\bar I) = \max \{j-i \mid \beta_{i,j}(S/I)\neq 0 \text{ and } i \geq c \}$.

    Let $x^B$ be a minimal generator of $\bar I$. Then there exist $\lambda_1 ,\ldots, \lambda_c \in [0,1]$ with $\sum_{i=1}^{c}\lambda_i=1$ such that $ \sum_{i=1}^c \lambda_iA_i \preccurlyeq B \preccurlyeq \sum_{i=1}^c A_i$.
    Therefore, $\LCM\{x^A \mid x^A \in \mathcal G(\bar I)\} \preccurlyeq \sum_{i=1}^c A_i$. By Remark \ref{rmk:lcmIsMaxDegree}, we see that if $\beta_{i,C}(S/\bar I) \neq 0$ for some $C \in \mathbb Z_{\geq 0}^n$, then $C \preccurlyeq \sum_{i=1}^c {A_i}$. 

    If $\beta_{i,C}(S/\bar I)\neq 0$ for some $i \geq c$ and $C \in \mathbb Z_{\geq 0}^n$, then we have $\deg(x^C)-c \leq \deg\left(x^{\sum_{i=1}^c {A_i}}\right) -c = \reg(S/I)$. This shows that $\reg(S/\bar I) \leq \reg(S/I)$.
\end{proof}

\section{Gorenstein Monomial Ideals of Height 3}\label{sec:Gor}
It is well-known that the Gorenstein ideals of height 2 are complete intersection ideals. Therefore, for such monomial ideals,  Conjecture \ref{conj:integral-closure} holds, and we also have $\beta_i(S/I)\leq \beta_i(S/\bar I)$ for all $0\leq i \leq 2$.

In this section, we prove \Cref{conj:integral-closure} for all Gorenstein monomial ideals of height $\leq 3$.
We first show that for Gorenstein monomial ideals of codimension $3$, the inequality $\mu(I) \leq \mu(\bar I)$ holds. As a consequence, we obtain $\beta_i(S/I) \leq \beta_i(S/\bar I)$ for all $i$. Before going to the proof, in the following remark, we note a result about the structure of the minimal generating set of Gorenstein monomial ideals of height three.

\begin{rmk}\label{rmk:Kamoithm}
    {\rm (\cite[Theorem 0.1]{Ka95}) Let $I$ be a monomial ideal with $\mu(I)=m$. Suppose $s=(m+1)/2$ and $\nu: \mathbb Z \to  
    \{1,\ldots, m\}$ be a map such that $i \equiv \nu(i) ({\rm mod\ } m)$ for $i \in \mathbb Z$. Then the following are equivalent:
    \begin{enumerate}
        \item $I$ is a Gorenstein monomial ideal of codimesion three.
        \item $m$ is odd, and there exist $m$ pairwise coprime monomials $x^{B_1},\ldots, x^{B_m}$ such that 
        \[ I= \left\langle \left\{ \prod\limits_{k=1}^{s-1} x^{B_{\nu(i+k)} }\mid 1 \leq i \leq m \right\} \right\rangle.\]
    \end{enumerate}
    }
\end{rmk}
The same result as in Remark \ref{rmk:Kamoithm} is proved independently in \cite{BH95}.

    \begin{lemma}\label{lem:mingensforgor} If $I$ is a Gorenstein ideal of codimension three, then every minimal generator of $I$ is a minimal generator of $\bar I$. In particular, $\mu(I) \leq \mu(\bar I)$.
    \end{lemma}
    \begin{proof}
        From Remark \ref{rmk:Kamoithm}, we see that for some odd $m$, there exist $m$ pairwise coprime monomials $x^{B_1},\ldots, x^{B_m}$ such that 
                 \[ I= \left\langle \left\{ \prod\limits_{k=1}^{s-1} x^{B_{\nu(i+k)} }\mid 1 \leq i \leq m \right\} \right\rangle.\] 
        Denote $\prod\limits_{k=1}^{s-1} x^{B_{\nu(i+k)}}$ by $x^{C_i}$. Since $x^{B_1},\ldots, x^{B_m}$ are pairwise coprime, from the definition of $x^{C_i}$ we observe that if $i\neq j$, then for some $k$, $x_k\nmid x^{C_i}$ and $x_k \mid x^{C_j}$ for some $k$.

        We now show that each $x^{C_i}$ is a minimal generator of $\bar I$. If possible, suppose that for some $1\leq i \leq m$ we have $C_i> \sum\limits_{j=1}^{m} r_j C_j$ for some $r_j \in \mathbb Q_{\geq 0}, 0\leq r_i\leq 1$. with $\sum\limits_{j=1,\\ j\neq i}^{m} r_j =1$. Then by the observation from the previous paragraph, we must have $r_j=0$ for all $j \neq i$. This contradiction shows that $x^{C_i}$ is a minimal generator of $\bar I$, and hence the proof is complete.
    \end{proof}
    
    \begin{prop}\label{prop:gorIntegralClosure}
        Let $I$ be a Gorenstein monomial ideal of codimension three. Then $\beta_i(S/I) \leq \beta_i(S/\bar I)$ for all $i$.
    \end{prop}
    \begin{proof}
        Note that by Lemma \ref{lem:mingensforgor}, $\mathcal G(I)\subseteq \mathcal G(\bar I)$. If $\mathcal G(I)=\mathcal G(\bar I)$, then there is nothing to prove. Hence, we assume that $m:=\mu(I)<\mu(\bar I)=:m'$. Note that since $\hgt(I)=\hgt(\bar I)$, and $S$ is Cohen--Macaulay, we get $\dim(S/\bar I)=\dim(S/I)=3$, and hence $\depth(S/\bar I) \leq 3$. By the Auslander-Buchsbaum formula, we get $\pdim_S(S/\bar I) \geq \pdim_S(S/I)$. By the structure theorem for resolutions of Gorenstein ideals of height three (see \cite{BE77}), we have $\beta_1(S/I)=\beta_2(S/I)=m, \beta_3(S/I)=1$. Therefore, to complete the proof, it is enough to show that $\beta_2(S/\bar I)\geq m$. 

        Let $\mathcal G(\bar I)=\{f_1,\ldots,f_{m'}\}$, and $\phi:S^{m'}\to S$ be given by $\phi(e_j)=f_j$ for all $j$. Then we have $f_ie_1-f_1e_i \in \Omega_2(S/\bar I)$ for all $i \geq 2$. We claim that these $m'-1$ elements cannot be spanned by $\leq m'-2$ elements. If the claim is false, then the same holds for these elements, when viewed as elements of $S^{m'}_{(0)}$ as an $S_{(0)}$-module. Now, note that $S_{(0)}$ is a field, and clearly the set $\{f_je_1-f_1e_j\}\subseteq S^{m'}_{(0)}$ is $S_{(0)}$-linearly independent. Thus, $\{f_je_1-f_1e_j\mid 2\leq j \leq m'\}\subseteq \Omega_2(S/\bar I)$ cannot be generated by less than $m'-1$ elements. This proves the claim. Therefore, we get $\beta_2(S/\bar I)\geq m'-1\geq m$. This completes the proof.
    \end{proof}

\begin{rmk}{\rm
    Let $S/I$ be Gorenstein. Note that from the proof of Lemma \ref{lem:mingensforgor}, we see that every minimal generator of $I$ is also a minimal generator of $\bar I$. Using this, and the proof of Proposition \ref{prop:gorIntegralClosure}, it follows that if $J$ is a ideal such that $I \subseteq J \subseteq \bar I$, then for all $i\geq 0$ we have $\beta_i(S/J) \leq \beta_i(S/\bar I)$.
}\end{rmk}

\begin{thm}\label{thm:Gor-thm}
    Let $I$ be a Gorenstein monomial ideal of codimension $3$. Then $\reg(\bar I) \leq \reg(I)$.
\end{thm}
\begin{proof}
    Let $\mu(I)=m$. By Remark \ref{rmk:Kamoithm}, there exist $m$ pairwise coprime monomials $x^{B_1},\ldots, x^{B_m}$ such that 
        \[ I= \left\langle \left\{ \prod\limits_{k=1}^{s-1} x^{B_{\nu(i+k)} }\mid 1 \leq i \leq m \right\} \right\rangle.\]
    Let $d_i=\deg\left(x^{B_i}\right)$. Then by the proof of \cite[Theorem 6.1]{BH95}, we see that $S/I$ has a minimal free resolution of the form 
         \[ 0 \to S(-\gamma) \to \bigoplus\limits_{i=1}^m S(-\beta_i) \to \bigoplus\limits_{i=1}^m S(-\alpha_i) \to S \to S/I \to 0,\]
         where $\gamma =  \sum_{i=1}^m d_i$. Since $\pdim(S/I)=3=\hgt(I)=\hgt(\ann(S/I))$, by Remark \ref{rmk:upToCodim}, we get that $\reg(S/I)=-3+\sum_{i=1}^m d_i$.
     Again, since $\hgt(\bar I)=3$, we get $$\reg(S/\bar I)= \max\{j-i \mid \beta_{i,j}(S/\bar I) \neq 0 \text{ and } i \geq 3 \}.$$ Since $\LCM\{x^A \mid x^A \in \mathcal G(\bar I)\}:=x^B$ has degree $\sum_{i=1}^m d_i$, by Remark \ref{rmk:lcmIsMaxDegree}, $\beta_{i,C}(S/I)=0$ if $C \not \preccurlyeq B$.   
     Thus if $\beta_{i,C}(S/\bar I)\neq 0$ for some $i \geq 3$ and $C \in \mathbb Z_{\geq 0}^n$, then we have $-3+ \deg(x^C) \leq -3+ \sum_{i=1}^m d_i = \reg(S/I)$. This shows that $\reg(S/\bar I) \leq \reg(S/I)$.
\end{proof}

    If $I$ is a monomial ideal of $\sk[x,y]$, or a complete intersection monomial ideal, or a Gorenstein monomial ideals of height $3$, then we have seen the inequality $\beta_i(S/I)\leq \beta_i(S/\bar I)$ holds for all $i\geq 0$. However, this is false in general, as the following example shows:

\begin{example}\label{ex:counterExamplein3D}
    {\rm 
    Consider the ideal
    \[I= \langle x^5, y^5, z^5, x^4y^2, x^4yz,x^4z^2, x^3y^3, x^3y^2z, x^3yz^2, x^3z^3, x^2y^4, x^2y^3z,\] \[\qquad \qquad x^2y^2z^2,x^2yz^3,x^4z^4,xy^4z,xy^3z^2,xy^2z^3,xyz^4,y^4z^2,y^3z^3,y^2z^4\rangle\]
    in $\sk[x,y,z]$ generated by $22$ elements, i.e., $\beta_1(S/I)=22$. Then we have $\bar I= \langle x,y,z\rangle^5$, and hence $\beta_1(S/ \bar I) =21$. 
}\end{example}    

\section{Stable Monomial Ideals}\label{sec:Stable}
\begin{notation}
   {\rm For a nonconstant monomial $x^A\in S$, by $m(x^A)$ we denote the largest index $i$ such that $x_i \mid x^A$.}
\end{notation}

\begin{defn}
     A monomial ideal $I$ is called a \emph{stable ideal} if for every monomial $x^A \in I$ and every $i< m(x^A)$, we have $x_i (x^A/x_{m(x^A)}) \in I$. It is called \emph{strongly stable} if for every monomial $x^A$ in $I$, we have $x_i(x^A/x_j)\in I$ whenever $x_j \mid x^A$ for some $j>i$.
\end{defn}

    The following is a well-known result due to Eliahou--Kervaire \cite{EK90} about the graded Betti numbers of stable ideals.

\begin{rmk}[\cite{HH11}, Corollary 7.2.3]\label{rmk:Eliahou-Kervaire}{\rm  Let $I\subseteq S$ be a stable ideal. Then 
    \begin{enumerate}[label={\rm (\alph*)}]
        \item $\beta_{i,i+j}(I)= \sum\limits_{u \in \mathcal G(I)_j} \binom{m(u)-1}{i}$.
        \item $\pdim_S(S/I)=\max\{m(u) \mid u \in \mathcal G(I)\}$.
        \item $\reg(I)= \max\{\deg(u) \mid u \in \mathcal G(I)\}$.
    \end{enumerate}   
}\end{rmk}
We shall make use of Remark \ref{rmk:Eliahou-Kervaire}(c) in  Theorem \ref{thm:Stable-ideals} below. Before doing so, we will show that if $I$ is stable, then so is $\bar I$. In \cite{GW15}, the authors prove the same result for strongly stable ideals along with a few more classes of ideals. However, their result does no seem to cover a slightly more general class of stable ideals.

\begin{rmk}[\cite{GW15}, Theorem 2.1] \label{rmk:intclosureofboreltype}{\rm Let $I$ be a monomial ideal of $S$. If $I$ is of
    Borel type (strongly stable, Borel-fixed, lexsegment, or universal lexsegment respectively),
    then $\bar I$ is also monomial of Borel type (strongly stable, Borel-fixed, lexsegment, or universal lexsegment respectively).
    }\end{rmk}

\begin{prop}\label{prop:closureIsStable}
    Let $I$ be a stable ideal. Then $\bar I$ is stable.
\end{prop}
\begin{proof}
Let $x^B\in \bar I$. Then there exists $r \in \mathbb N$ and monomials $x^{A_1},\ldots,x^{A_r} \in I, \ x^D \in \sk[x_1,\ldots,x_n]$ such that 
              $$(x^B)^r = \left(\prod_{j=1}^r x^{A_{j}}\right) x^D.$$  
Let $m(x^B)=t$. Fix $i<t$. We show that $x_i(x^B/x_t)\in \bar I$.
Since $x_t \mid x^B$, we have $x_t^r \mid x^{rB}$. Let $s_{r+1}$ be the largest integer such that $x_t^{s_{r+1}}\mid x^D$, and let $s_j$ be the largest integer such that $x_t^{s_j}$ divides $x^{A_j}$. Then $s_0+s_1+s_2+\cdots+s_r\geq r$. Since $I$ is stable, we have that $x_i^{s_j}x^{A_j}/x_t^{s_j} \in I$.

If $s_{r+1}\geq r$, then we are done since
$$(x_ix^B/x_t)^r = \left(\prod_{j=1}^r x^{A_{j}}\right) x_i^r(x^D/x_t^r).$$
If $s_{r+1}<r$, then let $j_0$ be the integer such that for some $0 \leq u<s_{j_0}$, we have $s_1+\cdots+s_{j_0}-u =r$. If $j_0\leq r$, then 
$$(x_ix^B/x_t)^r = \left(\prod_{j=1}^{j_0-1}  x_i^{s_j}x^{A_{j}}/x_t^{s_j}\right) \left( x_i^{s_{j_0}-u}/x_t^{s_{j_0}-u}\right)\left(\prod_{j=j_0+1}^r x^{A_{j}}\right) x^D.$$
If $j_0=r+1$, then 
$$(x_ix^B/x_t)^r = \left(\prod_{j=1}^{r}  x_i^{s_j}x^{A_{j}}/x_t^{s_j}\right) x_i^{s_{r+1}-u}x^D/x_t^{s_{r+1}-u}.$$
In either case, we get $(x_ix^B/x_t)^r \in I^r$, and hence $x_ix^B/x_t \in \bar I$, as desired. This shows that $\bar I$ is stable. 
\end{proof}

\begin{thm}\label{thm:Stable-ideals}
    Let $I$ be a stable ideal. Then $\reg(\bar I)\leq \reg(I)$.
\end{thm}
\begin{proof}
    Let $\mathcal G(I)=\{ x^{A_1}, \ldots, x^{A_k}\}$ denote the unique minimal monomial generating set of $I$, and let $N=\max\{\deg(x^A)\mid x^A \in \mathcal G(I)\}$. Note that by Proposition \ref{prop:closureIsStable}, we know that $\bar I$ is stable and hence by Remark \ref{rmk:Eliahou-Kervaire}, we have $\reg(R/\bar I)= \max\{\deg(x^A)\mid x^A \in \mathcal G(\bar I)\}$. Therefore, to prove the proposition it is enough to show that $\bar I$ has no minimal monomial generator in degree $\geq N+1$. 
    Let $x^B\in \mathcal G(\bar I)$. Then there exists $r \in \mathbb N$ such that 
              $$(x^B)^r = \left(\prod_{i=1}^r x^{A_{j_i}}\right) x^D.$$
    Suppose, if possible, $\deg(x^B)\geq N+1$. Then from the above equation we must have $\deg(x^D) \geq r$. Let $t$ be the largest index such that $x_t \mid x^B$. Then $x_t^{r}\mid x^{rB}$, and hence $x_t^r \mid \left(\prod_{i=1}^r x^{A_{j_i}}\right) x^D$. 
    If $x_t^r \nmid x^D$, then for some $j<t$, we have that $x_j \mid x^D$, and for some $1\leq j'\leq r$, we have that $x_t\mid x^{A_{j'}}$. Using the fact that $\bar I$ is a stable ideal, we replace $x^{A_{j'}}$ by $x_j x^{A_{j'}}/x_t\in I$ and $x^D$ by $x_tx^D/x_j$. Continuing this process, if necessary, in at most $r$ steps we get that $(x^B)^r = x^A x^{D'}$, where $x^A \in I^r$ and $x_t^r \mid x^{D'}$.
   Thus, cancelling $x_t^r$ from $x^{rB}$ and $x^{D'}$, we see that 
   $(x^{B-e_t})^r \in I^r$, showing $x^{B-e_t}\in \bar I$. This is a contradiction to the assumption that $x^B\in \mathcal G(\bar I)$. Hence, $\deg(x^B)\leq N$, and the proof is complete.
\end{proof}

\section{Possible directions and list of some more special classes of ideals}\label{sec:misc}

As mentioned in the Introduction, the lack of general techniques, along with the complexity of computing both the integral closure and the regularity, makes Conjecture~\ref{conj:integral-closure} difficult to tackle. In the present article, we have made crucial use of the geometry of the exponent sets of $I$ and $\bar{I}$ to extract information about regularity. Keeping this in mind, and focusing on the case of \emph{monomial ideals}, we list some classes of ideals for which Conjecture~\ref{conj:integral-closure} could be tested next:
\noindent
\begin{enumerate}[label={(\alph*)}]
    \item \textbf{Grade 2 perfect ideals:} Similar to the case of height 3 Gorenstein ideals, the minimal free resolution of grade 2 perfect ideals is also known, and is given by the Hilbert--Burch theorem. In addition, information about the generators of $I$ is also known (see  \cite[Remark 6.3]{BH95}).
    
    \item \textbf{Ideals with a linear resolution / equigenerated ideals with linear quotients:} For such ideals, the conjecture reduces to showing that $\bar{I}$ also has a linear resolution.
    
    \item \textbf{Powers of monomial ideals having a linear resolution.}
    
    \item \textbf{Equigenerated ideals:} For such ideals, using the description of the integral closure in terms of the Newton polyhedron or the convex hull of the exponent set, it is reasonable to expect that $\bar{I}$ is also equigenerated (since the exponent set of minimal generators of $I$ lies on a single hyperplane in $\mathbb{R}^n$).
    
    \item \textbf{Ideals arising from operations on two ideals, or those defining special types of rings:} For instance, fiber products, whose minimal resolutions have been studied in various contexts.
    
    \item \textbf{(Special types of) Cohen--Macaulay ideals:} Note that by \Cref{rmk:upToCodim}, if $I \subseteq S = \sk[x_1, \ldots, x_n]$ is Cohen--Macaulay with $\hgt(I) = p$, then we have $\pdim(S/I) = p$, and hence
    \[
    \reg(S/I) = \max\{ j - p \mid \beta_{p, j}(S/I) \neq 0 \}.
    \]
    In other words, the regularity is achieved at the ``last step'' in the free resolution of $S/I$. Thus, having this additional information about the resolution of $S/I$ could help in comparing $\reg(I)$ and $\reg(\bar{I})$. Note that complete intersection ideals (see \Cref{thm:CI-reg}), height 3 Gorenstein ideals (see \Cref{thm:Gor-thm}), and grade 2 perfect ideals mentioned in (a) above are examples of such special classes of Cohen--Macaulay ideals.
\end{enumerate}

\end{document}